\documentclass [a4paper,12pt]{article}
\usepackage[numbers]{natbib}
\usepackage{graphicx}
\usepackage[small]{subfigure,epsfig}
\usepackage{indentfirst}
\usepackage{amsmath,latexsym,enumerate}
\usepackage{amssymb}
\usepackage{amsfonts}
\usepackage{xcolor}
\usepackage{authblk}
\usepackage{amsthm}
\usepackage{appendix}

\numberwithin{equation}{section}

\theoremstyle{plain} \newtheorem{theorem}{Theorem}

\newtheorem{proposition}{Proposition}
\newtheorem{corollary}{Corollary}

\usepackage{geometry} 
\begin{document}

\title{On the geometric and analytical properties of the anharmonic oscillator}

\author[1]{Jaume Gin\'e}

\author[2,3,4]{Dmitry I. Sinelshchikov}

\affil[1]{Departament de Matem\`atica, Universitat de Lleida, Avda. Jaume II, 69; 25001 Lleida, Catalonia, Spain}
\affil[2]{Instituto Biofisika (UPV/EHU, CSIC), University of the Basque Country, Leioa E-48940, Spain}
\affil[3]{Ikerbasque Foundation, Bilbao 48013, Spain}
\affil[4]{HSE University, 34 Tallinskaya Street, 123458, Moscow, Russian Federation}

\maketitle

\begin{abstract}
Here we consider the anharmonic oscillator that is a dynamical system given by $y_{xx}+\delta y^{n}=0$. We demonstrate that to this equation corresponds a new example of a superintegrable two-dimensional metric with a linear and a transcendental first integrals. Moreover, we show that for particular values of $n$ the transcendental first integral degenerates into a polynomial one, which provides an example of a superintegrable metric with additional polynomial first integral of an arbitrary even degree. We also discuss a general procedure of how to construct a superintegrable metric with one linear first integral from an autonomous nonlinear oscillator that is cubic with respect to the first derivative. We classify all cubic oscillators that can be used in this construction. Furthermore, we study the Li\'enard equations that are equivalent to the anharmonic oscillator with respect to the point transformations. We show that there are nontrivial examples of the Li\'enard equations that belong to this equivalence class, like the generalized Duffing oscillator or the generalized Duffing--Van der Pol oscillator.
\end{abstract}

\section{Introduction}
Integrability of dynamical systems is a complex notion that has different definitions in different contexts (see, e.g. \cite{Goriely,Llibre_book,Zhang}). However, it can be suggested that there are two unifying properties of integrable systems. The first one is that for an integrable system one can completely describe global and local behaviour of its trajectories, and the second one is that a generic dynamical system is not globally integrable. As a consequence, construction and classification of integrable systems is a non-trivial problem. One of the possible approaches to this problem is based on establishing integrability of a dynamical system by connecting it to another, perhaps simpler, integrable system. Here we show that superintegrability of a two-dimensional Riemannian metric and integrability of a family of Li\'enard equations, which includes the Duffing and the generalized Duffing--Van der Pol oscillators, can be demonstrated by connecting them to what is probably the simplest nonlinear oscillator, namely the anharmonic oscillator. We also extend this idea and show how autonomous, cubic with respect to the first derivative oscillators, can be used in the construction of superintegrable two-dimensional Riemannian metrics.

In this work we consider the anharmonic oscillator
\begin{equation}\label{eq:anho_t}
  y_{xx}+\delta y^{n}=0.
\end{equation}
Here $y=y(x)$ and in \eqref{eq:anho_t} and in what follows we denote derivatives by lower indices, i.e. $y_{x}=dy/dx$, $y_{xx}=d^{2}y/dx^{2}$ and so on. Throughout this work we also assume that $n\neq0,1$ and $\delta\neq0$ are arbitrary parameters. We exclude $n=0$ and $n=1$ from the consideration because at these values of $n$ system \eqref{eq:anho_t} becomes a linear equation. Since at $n=-1$ both autonomous and non-autonomous first integrals of \eqref{eq:anho_t} have a special form, this case will be considered separately in \ref{sec:appendix2}. Therefore, without loss of generality, we change $\delta$ in \eqref{eq:anho_t} to $(n+1)\delta$ to obtain
\begin{equation}\label{eq:anho}
  y_{xx}+\delta (n+1)y^{n}=0.
\end{equation}

It is worth noting that for $n=-3$ equation \eqref{eq:anho_t} admits a three-dimensional Lie algebra, while for all $n$ but $n=-3$, $n=1$, $n=0$ the Lie algebra of symmetries of \eqref{eq:anho_t} is two-dimensional. If $n=0$ and $n=1$ equation \eqref{eq:anho_t} is linear and admits eight dimensional Lie algebra.

Let us also remark that the harmonic oscillator (i.e. \eqref{eq:anho} at $n=1$) can be used for the description of small-amplitude oscillations of e.g. springs that obey the Hooke's Law. This is an idealized situation that assumes that a system displaced from equilibrium responds with a restoring force whose magnitude is proportional to the displacement. However, it fails to describe different nonlinear phenomena that appear when finite amplitude motion is studied. For example, anharmonic oscillator \eqref{eq:anho} and its generalizations can be used for the description of oscillatory motion of a load supported by rubber like
springs (see, \cite{Beatty1988,Beatty1988a,Beatty2012} and references therein).

In a broader context, oscillatory behaviour is ubiquitous in a wide range of applications and very often the idealized linear approximation breaks down and fails to describe different nonlinear phenomena that appear for example in biology, chemistry or mechanics (see, \cite{Strogatz,Murray}). Therefore, understanding of the behaviour of trajectories of nonlinear oscillators, like the Li\'enard systems, is an important problem. One of possible solutions to this problem is establishing integrability, i.e. in the context autonomous oscillators, the existence of an autonomous first integral. This allows one to completely describe the behavior of trajectories of the corresponding oscillator in the phase plane.

In this work, we demonstrate that there is a non-trivial superintegrable two-dimensional Riemannian metric with a linear and a transcendental first integrals corresponding to \eqref{eq:anho}. Then, we generalize this construction for autonomous, cubic with respect to the first derivative, nonlinear oscillators. We also find all oscillators of this type that can be used for constructing certain superintegrable two-dimensional Riemannian metrics. In addition, we show that important for applications integrable Li\'enard equations belong to the equivalence class of \eqref{eq:anho} with respect to point transformations.

Recall the definitions of integrability and superintegrability for Hamiltonian systems (see, e.g.  \cite{Winternitz2002,Fasso2005,Winternitz2013}). A Hamiltonian system is integrable, if it possesses the same amount of functionally independent first integrals, which are in involution, as the systems' degrees of freedom. A Hamiltonian system is called superintegrable if it has more functionally independent first integrals than its degrees of freedom. Notice that the maximal amount of functionally independent first integrals of a Hamiltonian system is equal to $2m-1$, where $m$ is the number its degrees of freedom.  For example, a two-dimensional Hamiltonian system is called superintegrable if it possesses three functionally independent first integrals. In its turn, a Riemannian metric is called integrable or superintegrable if the corresponding Hamiltonian system for its geodesic flow is integrable or superintegrable, respectively \cite{Matveev2008,Matveev2011,Kruglikov2008,Matveev2019,Bolsinov2018}.

Construction and classification of superintegrable metrics are interesting problems due to their importance in mathematical physics and differential geometry \cite{Matveev2011,Winternitz2013}. In the former, superintegrable systems are of great interest due to a possibility of an analytical description of their solutions, which is quite rare for physically relevant systems of differential equations. In the latter, there are several open problems that are connected with integrable geodesic flows \cite{Matveev2019,Bolsinov2018}. For example, in \cite{Matveev2019} (see problem 10.4) it was pointed out that it is an interesting problem to construct examples of metrics that are not of constant curvature and are not Darboux-superintegrable, where all geodesics are explicitly known. Below, we demonstrate that the metric connected to \eqref{eq:anho} provides such an example.

Furthermore, typically, superintegrable systems with polynomial in momenta integrals are considered \cite{Winternitz2002,Matveev2011,Winternitz2013,Valent2015,Valent2017}. There also has been some interest in integrable two-dimensional Hamiltonian systems with rational integrals \cite{Kozlov2014,Bagderina2017,Agapov2020,Agapov2021}. However, to the best of our knowledge, superintegrable systems with transcendental first integrals have not been constructed and, hence, it is interesting to obtain such examples. On the other hand, it is known \cite{Matveev2016,Bolsinov2018} that a generic integrable metric does not have non-trivial (i.e. functionally independent of the Hamiltonian) polynomial first integrals. Therefore, polynomial integrability or superintegrability imposes non-trivial constraints on a Riemannian metric.

Here we show that the metric corresponding to \eqref{eq:anho} is superintegrable with a linear and a transcendental first integrals. Moreover, at certain values of $n$ the transcendental first integral degenerates into a polynomial one of an arbitrary even degree. As far as we know, this is a first example of a superintegrable metric with additional transcendental first integral. Moreover, the metric connected to \eqref{eq:anho} provides an interesting example of a geodesic flow with a non-trivial polynomial first integral of an arbitrary even degree.


To show superintegrability of a metric connected to \eqref{eq:anho} we use two first integrals of \eqref{eq:anho}. The first one is the obvious autonomous integral
\begin{equation}\label{eq:anho_I_1}
  I_{1}=y_{x}^{2}+2\delta y^{n+1}.
\end{equation}
The second one is non-autonomous first integral that is
\begin{equation}\label{eq:anho_I_2}
  I_{2}=x-\frac{y y_{x}}{y_{x}^{2}+2\delta y^{n+1}}{}_{2}F_{1}\left(\frac{n+3}{2n+2},1;\frac{n+2}{n+1}; \frac{2\delta y^{n+1}}{y_{x}^{2}+2\delta y^{n+1}}\right),
\end{equation}
where ${}_{2}F_{1}(a,b;c;\zeta)$ is the hypergeometric function. The derivation of \eqref{eq:anho_I_2} is presented in \ref{sec:appendix}, where we use the Euler integral representation of the hypergeometic function (see formula \eqref{eq:eq_a5a}) as its definition. The standard definition of the hypergeometric function through power series can be found in \cite{Olver,Abramowitz,Bateman}. It is also demonstrated in \ref{sec:appendix} that \eqref{eq:anho_I_2} is a Liouvillian function. We recall roughly speaking that Liouvillian functions comprise a set of functions including the elementary functions and their repeated primitives.
It is easy to see that \eqref{eq:anho_I_2} degenerates into a rational expression with respect to $y_{x}$ (see, e.g. \cite{Abramowitz,Olver}) if $n=-(2k+3)/(2k+1)$ or $(n+3)/(2n+2)=-k$, $k\in\mathbb{N}^{0}=\mathbb{N}\cup\{0\}$, where $\mathbb{N}=\{1,2,3,\ldots\}$ is the set of natural numbers.  With the help of \eqref{eq:anho_I_1} this rational first integral can be converted into a polynomial one with the degree $2k+2$. Indeed, at $n=-(2k+3)/(2k+1)$ we have
\begin{equation}\label{eq:eq9c}
I_{2}^{(k)}=x\left(y_{x}^{2}+2\delta y^{-\frac{2}{2k+1}}\right)^{k+1}-y y_{x}\sum\limits_{s=0}^{k}(-1)^{s}\binom{k}{s}\frac{(2\delta)^{s}s!}{\left(\frac{1}{2}-k\right)_{s}} y^{-\frac{2s}{2k+1}}\left(y_{x}^{2}+2\delta y^{-\frac{2}{2k+1}}\right)^{k-s},
\end{equation}
where $(a)_{s}=a(a+1)\ldots(a+s-1)$ is the Pochhammer symbol and $\binom{k}{s}$ denotes the binomial coefficients.

Below, we also use these first integrals to construct integrable families of the Li\'enard equations. The integrability in this context is an intrinsic property of a given autonomous two-dimensional dynamical system, in particular the Li\'enard system, to have a global autonomous first integral. Integrability problem consists in determining the existence and the functional class of an autonomous first integral for a given dynamical system. For example, a two-dimensional dynamical system is Liouvillian integrable when its autonomous first integral can be expressed in terms of the Liouvillian functions (see, \cite{S,Gine2012}).

Finding integrable Li\'enard equations is an important problem from both mathematical and applied points of view (see, e.g. \cite{Gine2010,Demina2018,Gine2020,Sinelshchikov2020,Gine2011,Gine2022,Demina2022,Sinelshchikov2022,Llibre2022,Sinelshchikov2023}). For example, for an integrable Li\'enard system one can determine existence or non-existence of isolated periodic trajectories (limit cycles) or families of periodic trajectories (periodic trajectories surrounding neutrally stable fixed points). This is of interest due to the connection with still unsolved problems of the existence of limit cycles and nonlinear centers in a given two-dimensional autonomous dynamical system \cite{Gine2010,Demina2018,Gine2020,Sinelshchikov2020,Gine2011,Gine2022,Demina2022,Zhang}. As far as an applied point of view is concerned, for an integrable Li\'enard system using a global first integral one can analytically describe all possible trajectories, i.e. completely understand dynamics that is governed by this system \cite{Sinelshchikov2020,Llibre2022,Sinelshchikov2022,Sinelshchikov2023}. In addition, this allows one to analytically study various bifurcations and special trajectories such as homoclinic and heteroclinic loops in a given integrable system. Here we show that several important integrable Li\'enard equations belong to the equivalence class of \eqref{eq:anho} with respect to the point transformations. For example, known integrable cases of the generalized Duffing oscillator, that has application in physics and biology, belong to it. Furthermore, we find that there are other physically and mechanically relevant oscillators among the equivalence class of \eqref{eq:anho}. In addition, in Section 3 we demonstrate that all autonomous cubic oscillators that are connected to superintegrable two-dimensional Riemannian metrics with one linear first integral are also integrable in the above sense, i.e. they have an explicit global first integral, which follows from their linearizability via nonlocal transformations. There are physically, mechanically and biologically relevant examples among them. This demonstrates the importance of our results in a broader context of nonlinear science and its applications.

The rest of this work is organized as follows. In the next section we deal with the metric that corresponds to \eqref{eq:anho}. Section 3 is devoted to the generalization of the results of Section 2 to an autonomous cubic nonlinear oscillator.  In Section 4 we construct Li\'enard equations that are equivalent to \eqref{eq:anho} and consider several important examples of such Li\'enard equations. In the last Section we briefly summarize and discuss our results.

\section{Superintegrable metric with a transcendental first integral}
We begin with some basic definitions that will be necessary for the construction of the metric corresponding to \eqref{eq:anho} (see, e.g. \cite{DNF,Carmo,Matveev2008,Dunajski2009}).

Geodesics for a smooth two-dimensional surface with the metric
\begin{equation}\label{eq:metric}
  ds^{2}=g_{11}(x,y)dx^{2}+2g_{12}(x,y)dxdy+g_{22}(x,y)dy^{2},
\end{equation}
are defined by
\begin{equation}\label{eq:geodesics_eq}
\begin{gathered}
  \ddot{x}^{i}+\Gamma^{i}_{jk}\dot{x}^{j}\dot{x}^{k}=0, \quad (x^{1},x^{2})=(x,y),
  \end{gathered}
\end{equation}
where
\begin{equation}\label{eq:Levi-Civita}
\begin{gathered}
\Gamma^{i}_{jk}=\frac{g^{il}}{2}\left(\frac{\partial g_{jl}}{\partial x^{k}}+\frac{\partial g_{kl}}{\partial x^{j}}-\frac{\partial g_{jk}}{\partial x^{l}}\right), \quad
  g^{ik}g_{kj}=\delta^{i}_{j}, \quad i,j,k=1,2.
  \end{gathered}
\end{equation}

Conversely, geodesics satisfy the following two-dimensional Hamiltonian system
\begin{equation}\label{eq:geodesics_H}
  \dot{x}^{i}=H_{p_{i}}, \quad \dot{p}^{i}=-H_{x_{i}}, \quad H=\frac{1}{2}\left(g^{11}p_{1}^{2}+2g^{12}p_{1}p_{2}+g^{22}p_{2}^{2}\right).
\end{equation}

Projection of \eqref{eq:geodesics_eq} on the $(x,y)$ plane is
\begin{equation}\label{eq:cubic_eq}
  y_{xx}+a_{3}(x,y)y_{x}^{3}+a_{2}(x,y)y_{x}^{2}+a_{1}(x,y)y_{x}+a_{0}(x,y)=0,
\end{equation}
with
\begin{equation}\label{eq:projection}
  a_{3}=-\Gamma_{22}^{1}, \quad a_{2}=\Gamma_{22}^{2}-2\Gamma_{12}^{1}, \quad a_{1}=2\Gamma_{12}^{2}-\Gamma_{11}^{1}, \quad a_{0}=\Gamma_{11}^{2}.
\end{equation}
Consequently, for every metric there exists an equation from \eqref{eq:cubic_eq} that is a projection of its geodesic flow. However, the converse of this statement is not true, i.e. not every equation from \eqref{eq:cubic_eq} corresponds to some metric \eqref{eq:metric}. Equations from \eqref{eq:cubic_eq} that correspond to a metric are called metrisable and necessary and sufficient conditions for metrisability of \eqref{eq:cubic_eq} were obtained in \cite{Dunajski2009}.

Using the Levi--Civita relation \eqref{eq:Levi-Civita}, we explicitly express correlations \eqref{eq:projection} in terms of the components of the metric tensor
\begin{align}\label{eq:pre_Liouville}
  a_{0} &=\frac{1}{2}\frac{2g_{11}g_{12,x}-g_{11}g_{11,y}-g_{12}g_{11,x}}{g_{11}g_{22}-g_{12}^{2}}, \nonumber \\
  a_{1} &=\frac{1}{2}\frac{2g_{11}g_{22,x}-3g_{12}g_{11,y}-g_{22}}{g_{11}g_{22}-g_{12}^{2}},\\
  a_{2 } &=\frac{1}{2}\frac{g_{11}g_{22,y}+3g_{12}g_{22,x}-2g_{12}g_{12,y}-2g_{22}g_{11,y}}{g_{11}g_{22}-g_{12}^{2}},  \nonumber \\
  a_{3} &=\frac{1}{2}\frac{g_{12}g_{22,y}+g_{22}g_{22,x}-2g_{22}g_{12,y}}{g_{11}g_{22}-g_{12}^{2}}. \nonumber
\end{align}
Thus, for any metric \eqref{eq:metric} (recall that $g_{ij}$ is positive-definite) we necessary obtain a projective equation in the from \eqref{eq:cubic_eq} with coefficients given by \eqref{eq:pre_Liouville}. Now suppose that we have four functions $a_{l}$, $l=\overline{0,3}$ and we need to find whether there are three functions $g_{11}$, $g_{12}$ and $g_{22}$ satisfying \eqref{eq:pre_Liouville}. Consequently, we can consider \eqref{eq:pre_Liouville} as an overdetermined system of equations for $g_{11}$, $g_{12}$, $g_{22}$ and if it is compatible, then the corresponding equation from \eqref{eq:cubic_eq} is metrisable.

With the help of the transformations proposed by R. Liouville \cite{Liouville,Matveev2008,Dunajski2009}
\begin{equation}\label{eq:Liouville_1}
  \psi_{1}=\Delta^{2} g_{11}, \quad \psi_{2}=\Delta^{2}g_{12}, \quad \psi_{3}=\Delta^{2}g_{22}, \quad \Delta=\psi_{1}\psi_{3}-\psi_{2}^{2}\neq0,
\end{equation}
system \eqref{eq:pre_Liouville} can be transformed into
\begin{equation}\label{eq:Liouville}
\begin{gathered}
  \psi_{1,x}=-\frac{2}{3}a_{1}\psi_{1}+2a_{0}\psi_{2},\\
  \psi_{3,y}=-2a_{3}\psi_{2}+\frac{2}{3}a_{2}\psi_{3},\\
  \psi_{1,y}+2\psi_{2,x}=-\frac{4}{3}a_{2}\psi_{1}+\frac{2}{3}a_{1}\psi_{2}+2a_{0}\psi_{3},\\
  \psi_{3,x}+2\psi_{2,y}=-2a_{3}\psi_{1}+\frac{4}{3}a_{1}\psi_{3}-\frac{2}{3}a_{2}\psi_{2}.
  \end{gathered}
\end{equation}
This is a linear (contrary to \eqref{eq:pre_Liouville} that is nonlinear) overdetermined system of four equations for three functions $\psi_{l}$, $l=1,2,3$ that was obtained by R. Liouville \cite{Liouville,Matveev2008,Dunajski2009}. An equation from \eqref{eq:cubic_eq} is metrisable if and only if there is a solution of \eqref{eq:Liouville} such that $\Delta$ does not vanish \cite{Liouville,Matveev2008,Dunajski2009}. The  modern variant of the proof of this statement is given in \cite{Matveev2008}. Notice also that geometric motivation for the change of variables \eqref{eq:Liouville_1} is provided in \cite{Matveev2008}. System \eqref{eq:Liouville} is called the Liouville system \cite{Matveev2008,Dunajski2009,Dunajski2018}. The compatibility conditions for \eqref{eq:Liouville} were constructed in \cite{Dunajski2009}. On the other hand, for a given equation from \eqref{eq:cubic_eq}, one can directly check whether there is a solution of the Liouville system that satisfies the condition $\Delta\neq0$.

It can be directly shown by obtaining integrability conditions for the Liouville system \eqref{eq:Liouville} that equation \eqref{eq:anho} is metrisable with the metric
\begin{equation}\label{eq:metric_aho}
  ds^{2}=\frac{1}{C_{1}^{2}(2\delta C_{1} y^{n+1}+C_{2})^{2}}\left[(2\delta C_{1} y^{n+1}+C_{2})dx^{2}+C_{1}dy^{2} \right],
\end{equation}
where $C_{1}\neq0$ and $C_{2}$ are arbitrary constants.

The corresponding Hamiltonian for the geodesics of \eqref{eq:metric_aho} is
\begin{equation}\label{eq:H_aho}
  H=(2C_{1}\delta y^{n+1}+C_{2})\left(C_{1}\frac{p_{1}^{2}}{2}+(2C_{1}\delta y^{n+1}+C_{2})\frac{p_{2}^{2}}{2}\right).
\end{equation}
It is clear that \eqref{eq:H_aho} has a cyclic coordinate $x$ and, hence, the linear first integral
\begin{equation}\label{eq:L_aho}
  L=p_{1}.
\end{equation}
First integrals \eqref{eq:anho_I_1} and \eqref{eq:anho_I_2} can be lifted via $y_{x}=H_{p_{2}}/H_{p_{1}}$ to be first integrals of \eqref{eq:H_aho} as follows
\begin{equation}\label{eq:R_ahp}
  R=\frac{(2C_{1}\delta y^{n+1}+C_{2})^{2}}{C_{1}^{2}}\frac{p_{2}^{2}}{p_{1}^{2}}+2\delta y^{n+1},
\end{equation}

\begin{equation}\label{eq:T_ahp}
\begin{gathered}
  T=x-\frac{y v}{1+2\delta y^{n+1}v^{2}}\,{}_{2}F_{1}\left(\frac{n+3}{2n+2},1;\frac{n+2}{n+1},\frac{1}{1+2\delta y^{n+1}v^{2}}\right),\\
v=\left(1+\frac{C_{2}}{2C_{1}\delta y^{n+1}}\right)\frac{p_{2}}{p_{1}}.
  \end{gathered}
\end{equation}

It is easy to see that \eqref{eq:R_ahp} is a function of $H$ and $L$, however, the Jacobi matrix for $H$, $L$ and $T$ has the full rank.

One can also demonstrate that metric \eqref{eq:metric_aho} is not flat since its Gaussian curvature is
\begin{equation}\label{eq:curvature_aho}
  K=(n+1)C_{1}^{2}\delta(C_{2}ny^{n-1}+\delta C_{1}(n-1)y^{2n}).
\end{equation}

With the help of \eqref{eq:eq9c} one can see that at $n=-(2k+3)/(2k+1)$, $k\in\mathbb{N}^{0}$ Hamiltonian \eqref{eq:H_aho} admits a first integral
\begin{equation}\label{eq:R_new}
\begin{gathered}
T_{k}=\left(\tilde{v}^{2}+2\delta y^{-\frac{2}{2k+1}}p_{1}^{2}\right)^{k+1}x-y\tilde{v}\sum\limits_{s=0}^{k}(-1)^{s}\binom{k}{s}\frac{(2\delta)^{s}s!}{\left(\frac{1}{2}-k\right)_{s}} y^{-\frac{2s}{2k+1}}\left(\tilde{v}^{2}+2\delta y^{-\frac{2}{2k+1}p_{1}^{2}}\right)^{k-s}p_{1}^{2s+1},\\
\tilde{v}=\left(\frac{C_{2}}{C_{1}}+2\delta y^{-\frac{2}{2k+1}}\right)p_{2}.
\end{gathered}
\end{equation}
This integral is a polynomial function of momenta with the degree $2k+2$. Consequently, we obtain an example of a two-dimensional integrable metric with a polynomial integral of an arbitrary even degree.

The case of $k=0$ or $n=-3$ leads to a quadratic first integral, which means that metric \eqref{eq:metric_aho} at $n=-3$ is a Darboux superintegrable one \cite{Matveev2011}. One can also use the results of \cite{Matveev2008} to demonstrate that \eqref{eq:metric_aho} is Darboux superintegrable if and only if $n=-3$. Indeed, assume that $n\neq0,1$. Then the dimension of the Lie algebra of projective vector fields of \eqref{eq:metric_aho} is 3 if and only if $n=-3$. Finally, in \cite{Matveev2008} it was shown that all Darboux superintegrable metrics have three-dimensional Lie algebra of projective vector fields.

It is known (see, e.g. \cite{Matveev2008}) that for every solution of projective equation \eqref{eq:cubic_eq} $y(x)$, the curve $(x,y(x))$ is, up to reparametrization, a geodesic of the corresponding metric. Therefore, with the help of the general solution of \eqref{eq:anho} one can explicitly construct all geodesics of \eqref{eq:metric_aho}. One can find the general solution of \eqref{eq:anho} by eliminating $y_{x}$ from \eqref{eq:anho_I_1} and \eqref{eq:anho_I_2}. Notice also that it is convenient to express $x$ as a function of $y$. Indeed, suppose that $I_{1}=C_{3}$ and $I_{2}=C_{4}$. Then geodesics of \eqref{eq:metric_aho} are formed by the curves
\begin{equation}\label{eq:geodesics_explicit}
(y,x(y))=\left(x,C_{4}\pm \frac{y}{C_{3}}\sqrt{C_{3}+2\delta y^{n+1}}\, {}_{2}F_{1}\left(\frac{n+3}{2n+2},1;\frac{n+2}{n+1}; \frac{2\delta}{C_{3}}y^{n+1}\right) \right),
\end{equation}
where $C_{3}\neq0$ and $C_{4}$ are arbitrary constants. If $C_{3}=0$, then geodesics of \eqref{eq:metric_aho} are given by
\begin{equation}\label{eq:geodesics_explicit_a}
 (y,x(y))=\left(x, C_{5}\pm\frac{1-n}{2\sqrt{-2\delta}}y^{\frac{1-n}{2}}\right),
\end{equation}
where $C_{5}$ is an arbitrary constant. Let us remark that, without loss of generality, one can assume $\delta$ equal to $-1$ in \eqref{eq:geodesics_explicit_a}, by redefining the constant $C_{1}$ in metric \eqref{eq:metric_aho}.

Let us also remark that by directly checking conditions from \cite{Matveev2011} for the existence of a superintegrable metric with a linear and a cubic integrals, one can show that at $n=-1/2$ \eqref{eq:T_ahp} degenerates into a cubic polynomial with respect to momenta. At any other value of $n\neq-1,0,1$ metric \eqref{eq:metric_aho} does not satisfy the conditions from \cite{Matveev2011}.

Now let us demonstrate two examples of Hamiltonian \eqref{eq:H_aho} with polynomials first integrals. Suppose that $k=1$ or $n=-5/3$, $\delta=1$, $C_{1}=1$ and $C_{2}=0$. Then, we obtain that the Hamiltonian
\begin{equation}\label{eq:H_p4}
  H=\frac{p_{1}^{2}}{y^{2/3}}+\frac{2p_{2}^{2}}{y^{4/3}},
\end{equation}
possesses the fourth order polynomial first integral
\begin{equation}\label{eq:I_p4}
T_1=\frac{x}{y^{2/3}}p_{1}^{4}-\frac{3}{y^{1/3}}p_{2}p_{1}^{3}+\frac{4x}{y^{2}}p_{2}^{2}p_{1}^{2}-\frac{2}{y}p_{1}p_{2}^{3}+\frac{4x}{y^{8/3}}p_{2}^{4}.
\end{equation}
If $k=2$ or $n=-7/5$ and $C_{1}$, $C_{2}$ and $\delta$ the same as in the previous case, we find that the Hamiltonian
\begin{equation}\label{eq:H_p6}
  H=\frac{p_{1}^{2}}{y^{2/5}}+\frac{2p_{2}^{2}}{y^{4/3}},
\end{equation}
admits the following sixth order polynomial first integral
\begin{equation}\label{eq:I_p6}
  T_{2}=\frac {x}{y^{6/5}}p_{1}^{6}-\frac {5 }{y^{1/5}}p_{2}p_{1}^{5}+\frac {6x}{y^{8/5}}p_{2}^{2}p_{1}^{4}-\frac {20}{3\,y^{3/5}}p_{2}^{3}p_{1}^{3}+\frac {12x}{y^{2}}p_{2}^{4}p_{1}^{2}-4\frac {4}yp_{2}^{5}p_{1}+\frac{8x}{y^{12/5}}p_{2}^{6}.
\end{equation}

In this section we have demonstrated that there is a superintegrable Riemannian metric with a linear and a transcendental first integrals. This transcendental first integral degenerates at certain values of the parameter $n$ into a polynomial one of an arbitrary even degree.

\section{General construction of a superintegrable metric from an autonomous equation from \eqref{eq:cubic_eq}}
In this Section we generalize the construction presented above for an autonomous equation from \eqref{eq:cubic_eq}.


Consider an autonomous equation from \eqref{eq:cubic_eq}
\begin{equation}\label{eq:cubic_autonomous}
  y_{xx}+k(y)y_{x}^{3}+h(y)y_{x}^{2}+f(y)y_{x}+g(y)=0,
\end{equation}
where $k$, $h$, $f$, $g$ are sufficiently smooth functions that do not vanish simultaneously.

Suppose that \eqref{eq:cubic_autonomous} is metrisable, i.e. the Liouville system \eqref{eq:Liouville} is compatible for coefficients of \eqref{eq:cubic_autonomous}. Suppose also that there exists such non-degenerate solution ( i.e. $\Delta=\psi_{1}\psi_{3}-\psi_{2}^{2}\neq0$) of the Liouville system \eqref{eq:Liouville} that $\psi_{1,x}=\psi_{2,x}=\psi_{3,x}=0$. Finally, we assume that \eqref{eq:cubic_autonomous} possesses an explicit representation for its non-autonomous first integral.

It is clear that under the above assumptions the Hamiltonian that corresponds to \eqref{eq:cubic_autonomous} does not depend explicitly on $x$ and, hence, has the cycling coordinate $x$ and the linear first integral $L=p_{1}$. When one lifts the non-autonomous first integral of \eqref{eq:cubic_autonomous}, the resulting first integral, say $F(x,y,p_{1},p_{2})$, of the Hamiltonian system will depend explicitly on $x$ in contrast to the Hamiltonian and $L$. Consequently, this first integral will be functionally independent of the Hamiltonian and the linear integral. Indeed, the corresponding Jacobi matrix will be
\begin{equation}\label{eq:Jacobi_m}
  J=\left(
      \begin{array}{cccc}
        0 & H_{y} & H_{p_{1}} & H_{p_{2}} \\
        F_{x} & F_{y} & F_{p_{1}} & F_{p_{2}} \\
        0 & 0 & 1 & 0 \\
      \end{array}
    \right).
\end{equation}
Now it is also clear that the autonomous first integral of \eqref{eq:cubic_autonomous}, will be a function of the Hamiltonian and the linear integral.

Consequently, one gets the following result:
\begin{proposition}
\label{th:th3}
  Suppose that an equation from \eqref{eq:cubic_autonomous} is metrisable and possesses an explicit non-autonomous first integral. Then, if there exists a non-degenerate solution of the Liouville system that depends only on $y$, this equation will lead to a superintegrable two-dimensional Riemannian metric.
\end{proposition}

All equations of form \eqref{eq:cubic_autonomous} that are metrisable with $\psi_{i,x}=0$, $i=1,2,3$ can be described as follows
\begin{theorem}
\label{th:th4}
  Equation \eqref{eq:cubic_autonomous} is metrisable with the metric that does not explicitly depend on $x$ if and only if one of the following correlations holds
  \begin{equation}\label{eq:mc_1}
    (I) \quad 27kg^{2}-9hfg+2f^{3}+9gf_{y}-9fg_{y}=0, \quad f\neq0, \quad k\neq0,
  \end{equation}
    \begin{equation}\label{eq:mc_2}
    (II)\quad 2f^{3}-9hfg+9gf_{y}-9fg_{y}=0,\quad f\neq0, \quad g\neq0,
  \end{equation}
    \begin{equation}\label{eq:mc_3}
    (III)\quad k=f=0, \quad g\neq0,
  \end{equation}
     \begin{equation}\label{eq:mc_4}
    (IV)\quad f=g=0, \quad k\neq0,
  \end{equation}
   \begin{equation}\label{eq:mc_5}
    (V)\quad f=g=k=0, \quad h\neq0
  \end{equation}
The solution of the Liouville system in each case can be found from the following relations
\begin{equation}\label{eq:ls_1}
\begin{gathered}
    (I)\quad \psi_{2}=\frac{2h\psi_{3}-3\psi_{3,y}}{6k}, \quad \psi_{1}=\frac{3g}{f}\psi_{2}, \\
    9fk \psi_{3,yy}-3B\psi_{3,y}-2(2kfh^{2}-6f^{2}k^{2}+3kfhy_{y}-Bh)\psi_{3}=0, \\
    B=fhk-9gk^2+3fk_{y}, \quad 2h\psi_{3}-3\psi_{3,y}\neq0,
    \end{gathered}
  \end{equation}
    \begin{equation}\label{eq:ls_2}
    (II)\quad \psi_{1}=\frac{3g\psi_{2}}{f}, \quad \psi_{2,y}=\frac{2f\psi_{3}-h\psi_{2}}{3}, \quad \psi_{3,y}=\frac{2h\psi_{3}}{3}, \quad \psi_{2}\neq0
  \end{equation}
    \begin{equation}\label{eq:ls_3}
    (III)\quad \psi_{1,y}=2g\psi_{3}-\frac{4h\psi_{1}}{3}, \quad \psi_{3,y}=\frac{2h\psi_{3}}{3}, \quad \psi_{2}=0, \quad \psi_{1}\psi_{3}\neq0,
  \end{equation}
     \begin{equation}\label{eq:ls_4}
    (IV)\quad \psi_{1,y}=-\frac{4h\psi_{1}}{3}, \quad \psi_{2,y}=-k\psi_{1}-\frac{h\psi_{2}}{3}, \quad \psi_{3,y}=\frac{2h\psi_{3}}{3}-2 k\psi_{2},
  \end{equation}
   \begin{equation}\label{eq:ls_5}
    (V)\quad \psi_{1,y}=-\frac{4h\psi_{1}}{3},\quad \psi_{2,y}=-\frac{h\psi_{2}}{3}, \quad \psi_{3,y}=\frac{2h\psi_{3}}{3}.
  \end{equation}
In all cases it should be assumed that $\psi_{1}\psi_{3}-\psi_{2}^{2}\neq0$.
\end{theorem}
\begin{proof}
  The proof is straightforward. Assuming that all $\psi_{i}$, $i=1,2,3$ do not depend on $x$, the Liouville system for \eqref{eq:cubic_autonomous} becomes an overdetermined system of ordinary differential equations. Its integrability conditions can be easily computed and are presented in \eqref{eq:mc_1}-\eqref{eq:mc_5}.
\end{proof}

It is worth noting that all families of equations described in Theorem \ref{th:th4} are completely integrable. The first family, given by \eqref{eq:mc_1}, is a particular case of cubic oscillators that can be linearized via the generalized nonlocal transformations (see \cite{Sinelshchikov2020}). The second family, defined by \eqref{eq:mc_2}, is equivalent to the harmonic oscillator $w_{\xi\xi}+\beta w_{\xi}+\alpha w=0$ at $\alpha=2\beta^{2}/9$ via the Sundman transformations (see \cite{Sinelshchikov2020a}). The case of \eqref{eq:mc_3} is also linearizable via the Sundman transformations to the $w_{\xi\xi}=0$. The remaining two cases can be integrated in  an obvious way. While this guarantees the existence of an explicit autonomous first integral, for the generating a superintegrable metric one needs an explicit expression for the non-autonomous first integral, finding which may be a non-trivial problem. Finally, we remark that classification presented in Theorem \ref{th:th4} is up to point transformations since they preserve metrisability. It is also worth noting that various particular cases of linearizable equations that are given by \eqref{eq:mc_1} and \eqref{eq:mc_2} are applicable in physics, mechanics and biology (see, e.g. \cite{Sinelshchikov2020,Sinelshchikov2020a} and references therein). This provides an interesting connection between integrable metrics with a linear first integral and  physically, mechanically and biologically relevant autonomous oscillators.

On the other hand, integrability of the equations described in Theorem \ref{th:th4} follows from the fact that the corresponding metric is integrable with a linear integral. One can find autonomous first integrals for equations defined by \eqref{eq:mc_1}-\eqref{eq:mc_5} by projecting a combination of the Hamiltonian and the linear first integral on the $(x,y)$ plane. In other words, suppose that a metric with $\psi_{1,x}=\psi_{2.x}=\psi_{3,x}$ corresponds to an equation from \eqref{eq:cubic_autonomous}. Then the corresponding Hamiltonian system for geodesics has two first integrals: the Hamiltonian and the linear first integral. Their combination can be projected to an autonomous first integral of an equation defined by \eqref{eq:mc_1}-\eqref{eq:mc_5}.

Now we deal with a canonical form of the equations described in Theorem \ref{th:th4}. It is well known (see \cite{Bolsinov1998}) that any two-dimensional Riemannian metric with a linear first integral can be transformed into
\begin{equation}\label{eq:metric_linear}
  ds^{2}=\lambda(y)(dx^{2}+dy^{2}).
\end{equation}
Here $\lambda\neq0$ is sufficiently smooth function.

The corresponding Hamiltonian is
\begin{equation}\label{eq:Hamiltonian_linear}
  H=\frac{p_{1}^{2}+p_{2}^{2}}{2\lambda}.
\end{equation}
An equation from \eqref{eq:cubic_eq} that corresponds to \eqref{eq:metric_linear} has the form
\begin{equation}\label{eq:linear_canonical}
  y_{xx}-\frac{\lambda_{y}}{2\lambda}y_{x}^{2}-\frac{\lambda_{y}}{2\lambda}=0.
\end{equation}
As a result, we have the following statement
\begin{theorem}
\label{th:th5}
  Any metrisable equation from families described in Theorem \ref{th:th4} can be transformed into \eqref{eq:linear_canonical} via point transformations with some particular function $\lambda$.
\end{theorem}

In this Section we have found equations from \eqref{eq:cubic_autonomous} that will generate a superintegrable two-dimensional metric once the explicit expression for their non-autonomous first integral is known. These equations can also be transformed into canonical form \eqref{eq:linear_canonical}.

\section{Li\'enard equations and equivalence class of the anharmonic oscillator}
In this section we construct a class of Li\'enard equations
\begin{equation}\label{eq:Lienard}
w_{\xi\xi}+f(w)w_{\xi}+g(w)=0,
\end{equation}
that is equivalent to \eqref{eq:anho} via point transformations
\begin{equation}\label{eq:point_trans}
  y=F(\xi,w), \quad x=G(\xi,w).
\end{equation}
We will show that this family of Li\'enard equations is Liouvillian integrable with a first integral that is expressed in terms of the hypergeometric function. The following statement holds:

\begin{theorem}
\label{th:th1}
  Suppose that $n\neq0,\pm1$. Then equation \eqref{eq:Lienard} is equivalent to \eqref{eq:anho} via \eqref{eq:point_trans} if one of the following correlations holds
  \begin{equation}\label{eq:criteria}
   \begin{gathered}
    (I)\quad f=\frac{2C_{1}n}{n-1}-\frac{3C_{1}MM_{ww}}{2M_{w}^{2}}, \quad
             g=\frac{C_{1}^{2}(n+1)M}{(n-1)M_{w}}-\frac{C_{1}^{2}M^{2}M_{ww}}{2M_{w}^{3}},
   \end{gathered}
  \end{equation}
    \begin{equation}\label{eq:criteria_fp}
   \begin{gathered}
    (II)\quad f=\alpha, \quad
             g=\frac{2(n+1)\alpha^{2}}{(n+3)^{2}}w+\delta w^{n},\quad n\neq-3,
   \end{gathered}
  \end{equation}
        \begin{equation}\label{eq:criteria_fp_n=-3}
   \begin{gathered}
    (III)\quad f=0, \quad
             g=-w-\delta w^{-3}, \quad n=-3.
   \end{gathered}
  \end{equation}

Equivalence transformations in all cases are given by
  \begin{equation}\label{eq:tr_1}
   (I)\quad F=\left(\frac{{\rm e}^{-2C_{1}\xi}(2M_{w}M_{www}-3M_{ww}^{2})}{4n\delta(n+1)M_{w}^{2}}\right)^{\frac{1}{n-1}}, \quad
    G={\rm e}^{C_{1}\xi}M,
  \end{equation}
  \begin{equation}\label{eq:tr_2}
   (II)\quad F=(n+1)^{-\frac{1}{n-1}}w{e}^{\frac{2\alpha \xi}{n+3}}, \quad  G=\frac{(n+3)}{\alpha (n-1)}{\rm e}^{-\frac{\alpha(n-1)\xi}{n+3}},\,n\neq-3,
  \end{equation}
    \begin{equation}\label{eq:tr_4}
   (III)\quad F=2{\rm e}^{\xi}w, \quad   G=\sqrt{2}{\rm e}^{2\xi}, \quad n=-3.
  \end{equation}

Here $M\neq0$ is a solution of the equation
\begin{equation}\label{eq:M_eq}
\begin{gathered}
  4nM_{w}^{2} \left(2 nM_{w}^{2}+MM_{ww}+2M_{w}^{2}-nMM_{ww} \right) M_{wwww}+4 \left( n-1\right) ^{2}MM_{w}^{2}M_{ww}^{2}-\\
  -4M_{ww}M_{w} \left(14 M_{w}^{2}{n}^{2} -3 M{n}^{2} M_{ww}+10 M_{w}^{2}n+3 MM_{ww}\right) M_{www}+\\+3 \left( 5 n+3 \right) M_{ww}^{3} \left(4\,M_{w}^{2}n+MM_{ww}-nMM_{ww} \right)=0
  \end{gathered}
\end{equation}
and $C_{1}\neq0$ is an arbitrary constant. Notice that families \eqref{eq:criteria}, \eqref{eq:criteria_fp} and \eqref{eq:criteria_fp_n=-3} are defined up to shifts and scalings in the independent and dependent variables.
\end{theorem}
\begin{proof}
One needs to consider two different cases separately, namely the case of $G_{w}\neq0$ and the case of $G_{w}=0$. The latter one corresponds to the fiber-preserving transformations.

Suppose that $G_{w}=0$. We also assume that $F_{w}H_{\xi}\neq0$, since otherwise transformations \eqref{eq:point_trans} degenerate. Substituting \eqref{eq:point_trans} into \eqref{eq:anho} we obtain an equation of the form $w_{\xi\xi}+\tilde{h}(\xi,w)w_{\xi}^{2}+\tilde{f}(\xi,w)w_{\xi}+\tilde{g}(\xi,w)=0$, whose coefficients are functions of $F$ and $G$.

Then we require that $\tilde{h}=\tilde{f}_{\xi}=\tilde{g}_{\xi}=0$. This leads to a system of three equations for two functions $F$ and $G$, which is
\begin{equation}\label{eq:th4_e1}
  F_{ww}=0, \quad \left(\ln\left\{\frac{F_{w}^{2}}{G_{\xi}}\right\}\right)_{\xi\xi}=0, \quad
 \left\{ \frac{G_{\xi}}{F_{w}}\left [\left(\frac{F_{\xi}}{G_{\xi}}\right)_{\xi}+\delta(n+1)F^{n}G_{\xi}\right]\right\}_{\xi}=0.
\end{equation}
Now we need to demonstrate that this system is compatible and find its solution. We differentiate the last equation from \eqref{eq:th4_e1} with respect to $y$ for three times and use the first two equations to exclude $F_{ww}$ and $F_{w\xi\xi}$ and their differential consequences. Consequently, we arrive at
\begin{equation}\label{eq:th4_e2}
\begin{gathered}
  \left\{\frac{G_{\xi\xi\xi}}{G_{\xi}}-\frac{3}{2}\frac{G_{\xi\xi}^{2}}{G_{\xi}^{2}} +2\delta  n (n+1) F^{n-1}G_{\xi}^{2} \right\}_{\xi}=0, \\
    \left\{\ln\left(F^{n-2}F_{w}G_{\xi}^{2}\right) \right\}_{\xi}=0, \quad \left\{\ln\left( F_{w}^{n-1}G_{\xi}^{2}\right) \right\}_{\xi}=0.
  \end{gathered}
\end{equation}
One can see that the further differentiation with respect to $w$ will not lead to any new relations on $F$ and $G$.

Now using the last equation from \eqref{eq:th4_e2} and second equations from \eqref{eq:th4_e2} and \eqref{eq:th4_e1} we get
\begin{equation}\label{eq:th4_e3}
\begin{gathered}
 FF_{\xi w}-F_{\xi}F_{w}=0, \quad FF_{\xi\xi}-F_{\xi}^{2}=0.
  \end{gathered}
\end{equation}
In order to obtain an equation for $G$ we substitute \eqref{eq:th4_e3} into the last equation from \eqref{eq:th4_e2} to obtain
\begin{equation}\label{eq:th4_e4}
\begin{gathered}
 2FG_{\xi\xi}+(n-1)G_{\xi}F_{\xi}=0.
  \end{gathered}
\end{equation}
The last equation from \eqref{eq:th4_e1} and the first equation from \eqref{eq:th4_e2} are automatically satisfied if one takes into account \eqref{eq:th4_e3} and \eqref{eq:th4_e4}.

Solving \eqref{eq:th4_e3} and \eqref{eq:th4_e4} together with the first equation from \eqref{eq:th4_e1} for the functions $F$ and $G$ we find
\begin{equation}\label{eq:th4_e5}
\begin{gathered}
F={\rm e}^{c_{1}\xi}\left(c_{2}y+c_{3}\right), \quad G=c_{4}+c_{5}{\rm e}^{-\frac{c_{1}(n-1)\xi}{2}},
  \end{gathered}
\end{equation}
and the corresponding equivalence class of \eqref{eq:anho} with respect to \eqref{eq:th4_e5} is
\begin{equation}\label{eq:th4_e6}
\begin{gathered}
w_{\xi\xi}+\frac{(n+3)c_{1}}{2}w_{\xi}+\frac{\delta(n+1)(n-1)^{2}c_{1}^{2}c_{5}^{2}}{4c_{2}}(c_{2}w+c_{3})^{n}+\frac{(n+1)c_{1}^{2}}{2}w+\frac{(n+1)c_{1}^{2}c_{3}}{2c_{2}}=0.
  \end{gathered}
\end{equation}
Here $c_{i},\, i=\overline{1,5}$ are some constants such that $c_{1}c_{2}c_{5}\neq0$. Since we define equivalence classes up to shifts and scalings, without loss of generality, we set $c_{3}=c_{4}=0$ and $c_{2}=(n+1)^{-1/(n-1)}$. Parameter $c_{5}\neq0$ corresponds to the scaling of the independent variable and is assigned below. Therefore, equivalence class \eqref{eq:th4_e6} is parametrized by $c_{1}$.

In order to simplify the presentation of our results, we set $c_{1}=2\alpha/(n+3)$ and $c_{5}=(n+3)(n-1)^{-1}\alpha^{-1}$ and use the parameters $\alpha$ and $\delta$ instead of $c_{1}$ and $c_{5}$. Thus, the coefficients at $w_{\xi}$ and $w^{n}$ are equal to $\alpha$ and $\delta$, respectively. As a result, we obtain family \eqref{eq:criteria_fp} and transformations \eqref{eq:tr_2}.

The above reparametrization degenerates at $n=-3$, when the coefficient at the first derivative in \eqref{eq:th4_e6} vanishes and all parameters can be removed from \eqref{eq:th4_e6}. Therefore, we present this case separately in \eqref{eq:criteria_fp_n=-3} and \eqref{eq:tr_4}, where we set $c_{2}=2$, $c_{5}=\sqrt{2}$ and $c_{1}=1$. Moreover, this case can be also taken apart because at $n=-3$ equation \eqref{eq:anho} and, hence, its equivalence class with respect to transformations \eqref{eq:point_trans}, possesses a three-dimensional Lie algebra of point symmetries.

Finally, the general case of $G_{w}\neq0$ can be treated in the same way. This completes the proof.
\end{proof}

\begin{corollary}
Any equation that is defined by \eqref{eq:criteria}, \eqref{eq:criteria_fp}, or \eqref{eq:criteria_fp_n=-3} at $n\neq-1$ possesses two first integrals that can be obtained from \eqref{eq:anho_I_1} and \eqref{eq:anho_I_2} by substituting \eqref{eq:point_trans} with \eqref{eq:tr_1}, \eqref{eq:tr_2} or \eqref{eq:tr_4} into \eqref{eq:anho_I_1} or \eqref{eq:anho_I_2}. These integrals are
\begin{equation}\label{eq:fi_r}
J_{1}=\left(\frac{F_{\xi}+F_{w}w_{\xi}}{G_{\xi}+G_{w}w_{\xi}}\right)^{2}+2\delta F^{n+1},
\end{equation}
\begin{equation}
\label{eq:fi_t}
J_{2}=G-\frac{F_{\xi}+F_{w}w_{\xi}}{G_{\xi}+G_{w}w_{\xi}}\frac{F}{J_{1}}\, {}_{2}F_{1}\left(\frac{n+3}{2n+2},1;\frac{n+2}{n+1}; \frac{2\delta F^{n+1}}{J_{1}}\right).
\end{equation}
\end{corollary}

Since transformations \eqref{eq:point_trans} depend explicitly on $\xi$, both integrals \eqref{eq:fi_r} and \eqref{eq:fi_t} can be non-autonomous, i.e. can depend explicitly on $\xi$. Therefore, any equation that is equivalent to \eqref{eq:anho} has two explicit non-autonomous first integrals. However, one can solve one of them for $\xi$, say $J_{1}$, and substitute the result into the other one, say $J_{2}$. As a result, we arrive to an autonomous first integral for systems described in Theorem \ref{th:th1}. Let us also recall that \eqref{eq:anho_I_2} is a Liouvillian function (see \ref{sec:appendix}) and, hence, \eqref{eq:fi_t} is also a Liouvillian function if $F$ and $G$ are Liouvillian functions. Consequently, the following statement holds:

\begin{theorem}
\label{th:th2}
 Any equation from \eqref{eq:Lienard} that is equivalent to \eqref{eq:anho} via \eqref{eq:point_trans} is Liouvillian integrable provided that $F$ and $G$ are Liouvillian functions. Otherwise, an equation from \eqref{eq:Lienard} that is equivalent to \eqref{eq:anho} is integrable with a transcendental first integral.
\end{theorem}

Since \eqref{eq:anho} possesses a two-dimensional Lie algebra of point symmetries at $n\neq-3$ and a three-dimensional algebra at $n=-3$, another corollary of Theorem \ref{th:th1} holds.
\begin{corollary}
Any equation from \eqref{eq:Lienard} that is equivalent to \eqref{eq:anho} at $n\neq0,1$ possesses either a two-dimensional or a three-dimensional Lie algebra of point symmetries.
\end{corollary}
Therefore, we see that any equation from \eqref{eq:criteria}, \eqref{eq:criteria_fp} and \eqref{eq:criteria_fp_n=-3} can be integrated via the classical Lie approach (see, e.g. \cite{Bluman}).

Now we consider several examples of Li\'enard equations that are equivalent to \eqref{eq:anho}.

\textbf{Example 1.} We  begin with a particular case of the generalized Duffing oscillator that is given by \eqref{eq:criteria_fp}:
\begin{equation}\label{eq:Duffing}
  w_{\xi\xi}+\alpha w_{\xi}+\frac{2(n+1)\alpha^{2}}{(n+3)^{2}}w+\delta w^{n}=0, \quad n\neq-3,-1,0,1.
\end{equation}
This family of equations was considered in \cite{Stachowiak2019,Demina2021}. The autonomous first integral in terms of the hypergeometric function for \eqref{eq:Duffing} was found in \cite{Stachowiak2019}, while in \cite{Demina2021} for $n\geq2$, $n\in\mathbb{N}$ it was shown that the Duffing oscillator is Liouvillian integrable if and only if it is given by \eqref{eq:Duffing}. However, it has not been demonstrated previously that integrability of \eqref{eq:Duffing} can be established from the equivalence to \eqref{eq:anho} and the existence of the corresponding two-dimensional Lie algebra of point symmetries.

Let us also remark that in \cite{Demina2021} the following equation was considered separately
\begin{equation}\label{eq:Duffing_n=2}
  w_{\xi\xi}+\alpha w_{\xi}-\frac{6\alpha^{2}}{25}w+\delta w^{2}=0.
\end{equation}
However, it can be transformed into \eqref{eq:Duffing} at $n=2$ by a simple shift $w\rightarrow w+6\alpha^{2}/(25\delta)$. Finally, we remark that the generalized Duffing oscillator is used in biology, mechanics, technics and other fields of science (see \cite{Stachowiak2019,Demina2021} and references therein.)

\textbf{Example 2.}
Let us suppose that
\begin{equation}
\label{eq:ex3_1}
  M=1+\frac{3\mu(m+1)}{2(m+2)w^{m}}, \quad m\neq-1,-2,-\frac{1}{3}
\end{equation}
and $\delta=-1$. Then from \eqref{eq:criteria} we get that the family of nonlinear oscillators
\begin{equation}\label{eq:ex3}
  w_{\xi\xi}+(w^{m}+\mu)w_{\xi}+\frac{2}{9(m+1)}w^{2m+1}+\frac{\mu}{m+2}w^{m+1}+\frac{(m+1)\mu^{2}}{(m+2)^{2}}w=0,
\end{equation}
is equivalent to \eqref{eq:anho} with $n=(1-m)/(3m+1)$ and $\delta=-1$.

Transformations \eqref{eq:point_trans} that map \eqref{eq:ex3} into \eqref{eq:anho} are given by
\begin{equation}\label{eq:ex3_3}
  F=\left(\frac{3 \sqrt{2}\mu m(m+1)}{(m+2)(3m+1) w^{m}}\right)^{\frac{3m+1}{2m}}{\rm e}^{-\frac{\mu(3m+1)\xi}{2m+4}}, \quad
  G={\rm e}^{-\frac{\mu m \xi}{m+2}}\left(\frac{3\mu(m+1)}{2(m+2)w^{m}}+1\right).
\end{equation}
Consequently, from Theorem \ref{th:th2} we obtain that \eqref{eq:ex3} is Liouvillian integrable. If one substitutes \eqref{eq:ex3_3} into \eqref{eq:fi_r} and \eqref{eq:fi_t}, one gets explicit expressions for the first integrals of \eqref{eq:ex3}.

On one hand, equation \eqref{eq:ex3} can be considered as a traveling wave reduction of the diffusion equation $u_{t}= u_{\zeta\zeta}+u^{m}u_{\zeta}+b_{1}u^{2m+1}+b_{2}u^{m+1}+b_{3}u$, where $w=u(\zeta-\mu t)$, $b_{1}=2/(9(m+1))$, $b_{2}=\mu/(m+1)$ and $b_{3}=(m+1)\mu^{2}/(m+2)^{2}$. This partial differential equation is used in physics for the description of diffusion and filtration problems \cite{Polyanin2011}. On the other hand, at $m=1$ \eqref{eq:ex3} is a cubic oscillator with linear damping, that has application in heat propagation problems \cite{Ervin1984}, and at $m=2$ \eqref{eq:ex3} is the generalized Duffing--Van der Pol oscillator, which is used in mechanics and physics \cite{Delignieres,Uzunov2014}.

\section{Conclusion and discussion}
In this work we have considered geometric and analytical properties of the anharmonic oscillator. We have demonstrated that \eqref{eq:anho} provides an example of a superintegrable two-dimensional Riemannian metric with a linear and a transcendental first integrals and a two-dimensional metric with a polynomial integral of an arbitrary even degree. In addition, we have found geodesics of metric \eqref{eq:metric_aho} in the explicit form. We have generalized this construction for an arbitrary nonlinear oscillator that is cubic with respect to the first derivative. We have explicitly described all such oscillators that can lead to a superintegrable metric with a linear first integral. Furthermore, we have constructed the family of Li\'enard equations that is equivalent to \eqref{eq:anho} with respect to the point transformations and demonstrated that there are examples of important nonlinear oscillators among this family.  In addition, we remark that the results of this work highlight a connection between integrable two-dimensional Riemannian metrics and nonlinear autonomous oscillators, wherein the latter can have applications in physics, biology and other fields of science. This also demonstrates a way of how the results on the integrability of two-dimensional metrics can be used in studying integrability of nonlinear oscillators and vice versa.

To conclude, our results show that different notions of integrability can be connected even through quite a simple dynamical system like \eqref{eq:anho}. While integrability of \eqref{eq:anho} alone is evident, both superintegrability of \eqref{eq:metric_aho} and integrability of the Li\'enard equations given by \eqref{eq:criteria} or \eqref{eq:criteria_fp} are not so obvious. Moreover, by connecting these Li\'enard equations with \eqref{eq:anho} we automatically demonstrate that in the generic case they admit a two-dimensional Lie algebra of point symmetries. Finally, let us remark that by excluding $y_{x}$ from \eqref{eq:anho_I_1} and \eqref{eq:anho_I_2} one can obtain the general solution of \eqref{eq:anho} in the explicit form (notice that we present $x$ as a function of $y$ in this way). With the help of transformations \eqref{eq:tr_1}, \eqref{eq:tr_2} and \eqref{eq:tr_4} we can map this general solution to the general solution of the corresponding Li\'enard equation. Therefore, it is easy to construct general solutions of the Duffing oscillators \eqref{eq:Duffing} or the generalised Duffing--Van der Pol oscillators \eqref{eq:ex3}, which, to the best of our knowledge, has not been done previously.

\section{Acknowledgement}
J.G. is partially supported by the Agencia Estatal de Investigac\'ion grant PID2020-113758GB-I00 and AGAUR grant number 2021SGR 01618. D.S. is partially supported by RSF grant 19-71-10003 (https://rscf.ru/en/project/19-71-10003/), Sections 2 and 3. Authors are grateful to anonymous referees and Editor for their valuable comments and suggestions.

\appendix
\section{Appendix A}
\label{sec:appendix}
Here we derive integral \eqref{eq:anho_I_2} for \eqref{eq:anho}. A non-autonomous first integral of \eqref{eq:anho} is a solution of the equation
\begin{equation}\label{eq:eq_a1}
  I_{x}+uI_{y}-\delta(n+1)y^{n}I_{u}=0,
\end{equation}
where $u=y_{x}$.
Suppose that $I=x+H(y,u)$. Then we get
\begin{equation}\label{eq:eq_a2}
  uH_{y}-\delta(n+1)y^{n}H_{u}=-1.
\end{equation}
In order to find a solution of \eqref{eq:eq_a2} we use the method of characteristics (see, e.g. \cite{Zhang,Arnold,Polyanin}). Characteristic equations of \eqref{eq:eq_a2} are
\begin{equation}\label{eq:eq_a2a}
  \frac{dy}{u}=\frac{du}{-\delta(n+1)y^{n}}=\frac{dH}{-1}.
\end{equation}
The general solution of \eqref{eq:eq_a2} is an arbitrary continuously differentiable function of two functionally independent first integrals of \eqref{eq:eq_a2a} \cite{Zhang,Arnold,Polyanin}. The first one can be obtained by integrating the equation that is formed by the first two terms of \eqref{eq:eq_a2a}
\begin{equation}\label{eq:eq_a3}
  u^{2}+2\delta y^{n+1}=c_{1}.
\end{equation}
Here $c_{1}$ is an integration constant. Finding the second first integral with the help of the first and last terms of \eqref{eq:eq_a2a} and taking into account \eqref{eq:eq_a3} we get that
\begin{equation}\label{eq:eq_a4}
  H=-\int \frac{dy}{\sqrt{ c_{1}-2\delta y^{n+1} }}.
\end{equation}
This quadrature can be computed as follows
\begin{equation}\label{eq:eq_a5}
\begin{gathered}
\int \frac{dy}{\sqrt{c_{1}-2\delta y^{n+1}}}=\frac{1}{\sqrt{c_{1}}(n+1)}\sqrt[n+1]{\frac{c_{1}}{2\delta}}\int \tau^{-\frac{n}{n+1}}(1-\tau)^{-\frac{1}{2}}d\tau=\\=
\frac{1}{\sqrt{c_{1}}(n+1)}\sqrt[n+1]{\frac{c_{1}}{2\delta}}\int\limits_{0}^{\tau} \zeta^{-\frac{n}{n+1}}(1-\zeta)^{-\frac{1}{2}}d\zeta=
\frac{\tau^{\frac{1}{n+1}}}{\sqrt{c_{1}}(n+1)}\sqrt[n+1]{\frac{c_{1}}{2\delta}}\int\limits_{0}^{1} v^{-\frac{n}{n+1}}(1-\tau v)^{-\frac{1}{2}}dv=\\
=\frac{1}{\sqrt{c_{1}}}\sqrt[n+1]{\frac{c_{1}}{2\delta}}\tau^{\frac{1}{n+1}}{}_{2}F_{1}\left(\frac{1}{2},\frac{1}{n+1};\frac{n+2}{n+1};\tau\right)=
\frac{y}{\sqrt{c_{1}}}{}_{2}F_{1}\left(\frac{1}{2},\frac{1}{n+1};\frac{n+2}{n+1};\frac{2\delta}{c_{1}}y^{n+1}\right).
\end{gathered}
\end{equation}
Here we use the Euler integral representation for the hypergeometric function \cite{Bateman,Abramowitz,Olver}
\begin{equation}\label{eq:eq_a5a}
  {}_{2}F_{1}(a,b;c;\tau)=\frac{\Gamma(c)}{\Gamma(b)\Gamma(c-b)}\int\limits_{0}^{1}v^{b-1}(1-v)^{c-b-1}(1-\tau v)^{-a}dv,
\end{equation}
and the following notations
\begin{equation}\label{eq:eq_a6}
  \tau=\frac{2\delta}{c_{1}}y^{n+1},\quad \zeta=\tau v.
\end{equation}
As a result, we get the expression for non-autonomous first integral of \eqref{eq:anho}
\begin{equation}\label{eq:eq_a7}
I_{2}^{'}=x-\frac{y}{\sqrt{y_{x}^{2}+2\delta y^{n+1}}}{}_{2}F_{1}\left(\frac{1}{2},\frac{1}{n+1};\frac{1}{n+1}+1; \frac{2\delta y^{n+1}}{y_{x}^{2}+2\delta y^{n+1}}\right).
\end{equation}
If one applies the linear transformation (see, e.g. \cite{Bateman,Abramowitz,Olver}) to \eqref{eq:eq_a7}, then one gets \eqref{eq:anho_I_2}.

One can see that the Euler representation for the hypergeometric function is obtained from the indefinite integral. Consequently, the hypergeometric function that is in \eqref{eq:eq_a7} or in \eqref{eq:anho_I_2} is a Liouvillian function.

\section{Appendix B}
\label{sec:appendix2}
In this appendix we consider \eqref{eq:anho_t} at $n=-1$, i.e.
\begin{equation}\label{eq:a2_1}
  y_{xx}+\delta y^{-1}=0.
\end{equation}
The autonomous firs integral of \eqref{eq:a2_1} is
\begin{equation}\label{eq:a2_2}
N_{1}=y_{x}^{2}+2\delta \ln y.
\end{equation}
In the same way as in \ref{sec:appendix} one can find the non-autonomous first integral of \eqref{eq:a2_1} as follows
\begin{equation}\label{eq:a2_3}
N_{2}=x+\sqrt{\frac{\pi}{2\delta}}y{\rm e}^{\frac{y_{x}^{2}}{2\delta}}\,\mbox{erf}\left\{\frac{y_{x}}{\sqrt{2\delta}}\right\}.
\end{equation}
Equation \eqref{eq:a2_1} is metrisable with the Riemannian metric
\begin{equation}\label{eq:a2_4}
  ds^{2}=\frac{1}{C_{1}^{2}(2\delta C_{1}\ln y+C_{2})^{2}}\left((2\delta C_{1}\ln y+C_{2})dx^{2}+C_{1}dy^{2}\right),
\end{equation}
where $C_{1}\neq0$ and $C_{2}$ are arbitrary constants.
The corresponding Hamiltonian system for geodesics is
\begin{equation}\label{eq:a2_5}
  H=(2\delta C_{1}\ln y+C_{2})C_{1}\left(C_{1} \frac{p_{1}^{2}}{2}+(2\delta C_{1}\ln y+C_{2})\frac{p_{2}^{2}}{2}\right).
\end{equation}
As in the case of \eqref{eq:anho}, integral $N_{1}$ lifted to \eqref{eq:a2_5} will be a combination of Hamiltonian \eqref{eq:a2_5} and the linear integral $L=p_{1}$. On the other hand, the integral $N_{2}$ will lead to an additional first integral of Hamiltonian \eqref{eq:a2_5}, which is functionally independent of $H$ and $L$.

\end{document}